\documentclass[reqno]{amsart}
\usepackage{amssymb}
\usepackage{aliascnt}
\usepackage{commath}
\usepackage{mathrsfs}
\usepackage{enumitem}
\usepackage{comment}
\usepackage[bookmarks=false,pdfstartview=FitH, pdfborder={0 0 0}, colorlinks=true,citecolor=blue, linkcolor=blue]{hyperref}
\usepackage{aliascnt} 

\theoremstyle{plain}
\newtheorem{theorem}{Theorem}[section]

\newaliascnt{conjecture}{theorem}

\aliascntresetthe{conjecture}

\newaliascnt{corollary}{theorem}

\aliascntresetthe{corollary}

\newaliascnt{lemma}{theorem}

\aliascntresetthe{lemma}

\newaliascnt{fact}{theorem}

\aliascntresetthe{fact}

\newaliascnt{claim}{theorem}

\aliascntresetthe{claim}

\newaliascnt{proposition}{theorem}
\newtheorem{proposition}[proposition]{Proposition}
\aliascntresetthe{proposition}

\theoremstyle{definition}

\newaliascnt{definition}{theorem}

\aliascntresetthe{definition}

\newaliascnt{example}{theorem}

\aliascntresetthe{example}

\theoremstyle{remark}

\newaliascnt{remark}{theorem}

\aliascntresetthe{remark}

\numberwithin{equation}{section}

\begin{document}
\title{Conjugate radius of open manifolds}
\author[J.~Ge]{Jian Ge*}
\address[Ge]{School of Mathematical Sciences, Laboratory of Mathematics and Complex Systems, Beijing Normal University, Beijing 100875, P. R. China.}
\email{jge@bnu.edu.cn}
\thanks{NSFC 12371049 and the Fundamental Research Funds for the Central Universities.}

\author[S.~Zhang]{Shimeng Zhang}
\address[Zhang]{School of Mathematical Sciences, Laboratory of Mathematics and Complex Systems, Beijing Normal University, Beijing 100875, P. R. China.}
\email{202531130043@mail.bnu.edu.cn}

\begin{abstract}
In this short note, we establish an upper bound for the conjugate radius of an open $n$-dimensional Riemannian manifold under a scalar curvature lower bound and a bottom-of-spectrum upper bound. As a consequence, if $\lambda_{0}(M)=0$ and scalar curvature $\ge n(n-1)$, then the conjugate radius $\le \pi$.
\end{abstract}
\maketitle
\section{Introduction}
The conjugate radius measures the extend to which geodesics remain free of conjugate points. It is naturally tied to the Jacobi equation: along a geodesic $\gamma$, Jacobi fields satisfy a second-order differential equation whose coefficients are determined by the sectional curvatures containing $\dot{\gamma}$. Thus, the conjugate radius appeaars at first to be governed by sectional curvature. A remarkable theorem of Leon Green shows, however, that it is also constrained by scalar curvature in the average sense. In 1963, Green \cite{Gre63} proved the following result:
\begin{theorem}[Green's Theorem]
	Let $M$ be an $n$-dimensional closed Riemannian manifold. If the conjugate radius of $M$ is bounded from below by ${\rm conj}(M)\geq a$, then
	\[
		\frac1{{\rm Vol}(M)}\int_M{\rm scal}(p)~{\rm dvol}_M(p)\leq n(n-1)\frac{\pi^2}{a^2},
	\]
	with equality if and only if $M$ has constant positive sectional curvature $\pi^2/a^2$.
\end{theorem}
For complete noncompact Riemannian manifolds, it remains open whether a scalar curvature lower bound alone forces an upper bound on the conjugate radius.

In 2022, Zhu \cite{Zhu22} established, among other results, the open manifold case under the additional assumption of non-negative Ricci curvature:
\begin{theorem}[Zhu]
	Let $M$ be an $n$-dimensional open Riemannian manifold satisfying
	\[
		\mathrm{scal}_M\geq n(n-1) \quad \text{and }\ {\rm Ric}_M\geq0,
	\]
	then the conjugate radius of $M$ satisfy:
	\[
		{\rm conj}(M)\leq\pi.
	\]
\end{theorem}

In 2025, Kwong \cite{Kwo25} obtained a related result under a different curvature together with the assumption of finite volume:
\begin{theorem}[Kwong]
	Let $M$ be an $n$-dimensional complete Riemannian with finite volume. Assume that
	\begin{enumerate}
		\item $\int_{SM}{\rm Ric}_M^-(p;v)~d\mu(p,v)$ is finite, where ${\rm Ric}_M^-(p;v)=\max\{-{\rm Ric}_M(p;v),0\}$ denotes the negative part of the Ricci curvature and $SM$ denotes the unit sphere bundle over $M$;
		\item $\int_M{\rm scal}(p)~{\rm dvol}_M(p)\geq n(n-1){\rm Vol}(M)$.
	\end{enumerate}
	Then the conjugate radius of $M$ satisfies
	\[
		{\rm conj}(M)\leq\pi,
	\]
	with equality if and only if $M$ has constant sectional curvature equal to $1$.
\end{theorem}

In this paper we prove:
\begin{theorem}[Main theorem]\label{main}
	Let $M$ be an $n$-dimensional open Riemannian manifold with ${\rm scal}_M\geq n(n-1)$ and $\lambda_{0}(M)<n$. Then the conjugate radius of $M$ satisfies
	\[
		{\rm conj}(M)\leq\frac\pi{\sqrt{1-\lambda_0(M)/n}},
	\]
	where $\lambda_{0}(M)$ denotes the bottom of the spectrum of the Laplacian on $M$. In particular, if in addition $\lambda_0(M)=0$, then
	\[
		{\rm conj}(M)\leq\pi.
	\]
\end{theorem}

Our estimate gives a finite conjugate-radius bound precisely in the range $\lambda_{0}<n$. This places the result in the broader contex of comparison results relating curvature assumptions to upper bounds for the bottom of the spectrum.

Also, we note that either ${\rm Ric}_M\geq0$ or ${\rm Vol}(M)<\infty$ implies that the bottom of the spectrum of $M$ vanishes, see proposition \ref{prop:ric} and \ref{prop:finiteVol}. Thus, our theorem may be viewed as a spectral version of Zhu's or Kwong's conjugate radius estimates.

\section{Bottom of the Spectrum.}
Let $M$ be a complete Riemannian manifold. Recall that the bottom of the spectrum of $M$ is defined by
	\[
	\lambda_0(M):=\inf\left\{\left.\frac{\int_M|\nabla f|^2~{\rm dvol}_M}{\int_M f^2~{\rm dvol}_M}\right|f\in C_c^\infty(M),f\not\equiv0\right\}.
	\]
The quantity inside the infimum is called the Rayleigh quotient. If $M$ is compact, then every smooth function on $M$ has compact support, so testing the Rayleigh quotient on constant functions yields $\lambda_0(M)=0$. In this case, the more meaningful quantity is the first nonzero eigenvalue $\lambda_1(M)$. By contrast, if $M$ is noncompact, $\lambda_0(M)$ may be vanish (for example, in Euclidean space) or be positive (for example, in hyperbolic space).

Now we show that either ${\rm Ric}_M\geq0$ or ${\rm Vol}(M)<\infty$ implies $\lambda_0(M)=0$.

\begin{proposition}\label{prop:ric}
	If $M$ is an open Riemannian manifold with ${\rm Ric}_M\geq0$, then $\lambda_0(M)=0$.
\end{proposition}
\begin{proof}
	Fix a point $p\in M$. For each $R>0$, define
	\[
		\varphi_R(x)=\psi\!\left(\frac{d(x,p)}R\right),
	\]
	where $\psi\in C^\infty(\mathbb R)$ satisfies
	\[
		\psi(t)=1 \quad \text{for } t\leq1,
	\]
	\[
		\psi(t)=0 \quad \text{for } t\geq2,
	\]
	\[
		0<\psi(t)<1 \quad \text{for } 1<t<2,
	\]
	and $|\psi'|\leq C$. Then $\varphi_R\in{\rm Lip}_c(M)$ with $\varphi_R\equiv 1$ on $B(p,R)$, $\varphi_R\equiv 0$ on $M\setminus B(p,2R)$, and $|\nabla\varphi_R|\leq C/R$ almost everywhere. Consequently,
	\[
		\int_{\{\nabla\varphi_R~{\rm exists}\}}|\nabla\varphi_R|^2~{\rm dvol}_M\leq\frac{C^2}{R^2}{\rm Vol}(B(p,2R)\backslash B(p,R))
	\]
	while
	\[
		\int_M\varphi_R^2~{\rm dvol}_M\geq{\rm Vol}(B(p,R)).
	\]
	Hence the Rayleigh quotient satisfies
	\[
		\frac{\int|\nabla\varphi_R|^2}{\int\varphi_R^2}\leq\frac{C^2}{R^2}\cdot
		\frac{{\rm Vol}(B(p,2R))-{\rm Vol}(B(p,R))}{{\rm Vol}(B(p,R))}
		\leq\frac{C^2}{R^2}\cdot\frac{{\rm Vol}(B(p,2R))}{{\rm Vol}(B(p,R))}.
	\]
	Since ${\rm Ric}_M\geq0$, the Bishop--Gromov volume comparison theorem implies that the function $r\mapsto{\rm Vol}(B(p,r))/\omega_nr^n$ is non-increasing. Therefore,
	\[
		\frac{{\rm Vol}(B(p,2R))}{(2R)^n}\leq\frac{{\rm Vol}(B(p,R))}{R^n}, \quad {\rm i.e.,}
		\frac{{\rm Vol}(B(p,2R))}{{\rm Vol}(B(p,R))}\leq 2^n.
	\]
	Substituting this estimate gives
	\[
		\frac{\int|\nabla\varphi_R|^2}{\int\varphi_R^2}\leq\frac{C^2 2^n}{R^2}\to 0 ~ \text{as } ~ R\to\infty.
	\]
	Since compactly supported Lipschitz functions can be approximated in $W^{1,2}$ by functions in $C^{\infty}_{c}(M)$, the above Rayleigh quotient estimates are admissible in the definition of $\lambda_{0}(M)$. Therefore $\lambda_0(M)=0$.
\end{proof}

\begin{proposition}\label{prop:finiteVol}
	If $M$ is an open Riemannian manifold with finite volume, then $\lambda_0(M)=0$.
\end{proposition}
\begin{proof}
	As in the proof of Proposition \ref{prop:ric},
	\[
		\frac{\int|\nabla\varphi_R|^2}{\int\varphi_R^2}
		\leq\frac{C^2}{R^2}\cdot\frac{{\rm Vol}(B(p,2R))-{\rm Vol}(B(p,R))}{{\rm Vol}(B(p,R))}.
	\]
	Since ${\rm Vol}(B(p,R))\to{\rm Vol}(M)<\infty$ as $R\to\infty$, we have
	\[
		{\rm Vol}(B(p,2R))-{\rm Vol}(B(p,R))\to{\rm Vol}(M)-{\rm Vol}(M)=0,
	\]
	therefore $\lambda_0(M)=0$.
\end{proof}

\section{Proof of the Main Theorem.}
Let
\[
	l:=\frac\pi{\sqrt{1-\frac{\lambda_0(M)}n}}.
\]
Suppose, to the contrary, that ${\rm conj}(M)>l$. Choose $a\in(l,{\rm conj}(M))$, and let $f\in C_c^\infty(M)$ be nonzero. For each $(p,v)\in SM$, let $\gamma(t)=\exp_p(tv)$ for $0\leq t\leq a$. Choose parallel vector fields $E_1(t), \cdots, E_{n-1}(t)$ along $\gamma$ such that
\[
	\{\gamma',E_1,\ldots,E_{n-1}\}
\]
forms an orthonormal frame. Define the variational vector field
\[
	V_i(t)=f(\gamma(t))\sin\left(\frac{\pi t}a\right)E_i(t).
\]
Then $V_i(0)=V_i(a)=0$. Since $a< {\rm conj}(M)$, the index form is nonnegative:
\begin{align*}
	0&\leq I(V_i,V_i)=\int_0^a\left(|\nabla_{\gamma'}V_i|^2-\langle R(\gamma',V_i)V_i,\gamma'\rangle\right)dt\\
 	&=\int_0^a\left[(f')^2\sin^2\left(\frac{\pi t}a\right)+2ff'\frac\pi{a}\sin\left(\frac{\pi t}a\right)\cos\left(\frac{\pi t}a\right)+f^2\frac{\pi^2}{a^2}\cos^2\left(\frac{\pi t}a\right)\right]dt\\
 	&\quad-\int_0^af^2\sin^2\left(\frac{\pi t}a\right)K(\gamma', E_i)dt,
\end{align*}
where, for brevity, we write $f$ for $f\circ\gamma$ and $f'$ for $(f\circ\gamma)'$. Here we use
\begin{equation*}
	K(X, Y):=\left\langle R(X, Y)Y, X \right\rangle
\end{equation*}
to denote the sectional curvature spanned by the orthonormal pair $X, Y$. Hence the Ricci curvature can be written as $\mathrm{Ric}(X, X)= \sum_{i=1}^{n-1} K(X, E_{i})$. Integration by parts yields
\begin{align*}
	&\quad\int_0^a2ff'\frac\pi{a}\sin\left(\frac{\pi t}a\right)\cos\left(\frac{\pi t}a\right)dt\\
	&=\int_0^a(f^2)'\frac\pi{a}\sin\left(\frac{\pi t}a\right)\cos\left(\frac{\pi t}a\right)dt\\
	&=\left.\left[f^2\frac\pi{a}\sin\left(\frac{\pi t}a\right)\cos\left(\frac{\pi t}a\right)\right]~\right|_{t=0}^{t=a}
		-\int_0^af^2\frac\pi{a}\left[\sin\left(\frac{\pi t}a\right)\cos\left(\frac{\pi t}a\right)\right]'dt\\
	&=\int_0^af^2\frac{\pi^2}{a^2}\left[\sin^2\left(\frac{\pi t}a\right)-\cos^2\left(\frac{\pi t}a\right)\right]dt.
\end{align*}
Substituting this gives
\[
	\int_0^a\left[(f')^2\sin^2\left(\frac{\pi t}a\right)+f^2\frac{\pi^2}{a^2}\sin^2\left(\frac{\pi t}a\right)
	-f^2\sin^2\left(\frac{\pi t}a\right)K(\gamma',E_i)\right]dt\geq0.
\]
Summing over $i=1,\cdots,n-1$ yields
\begin{equation}\label{eq2.1}
	\begin{aligned}
	&\quad\int_0^af^2\sin^2\left(\frac{\pi t}a\right){\rm Ric}(\gamma')dt\\
	&\leq(n-1)\frac{\pi^2}{a^2}\int_0^a f^2\sin^2\left(\frac{\pi t}a\right)dt
	+(n-1)\int_0^a\langle\nabla f,\gamma'(t)\rangle^2\sin^2\left(\frac{\pi t}a\right)dt.
	\end{aligned}
\end{equation}
Recall that, for each $t\in\mathbb R$, the geodesic flow is defined by:
\[
	\varphi_t:TM\to TM, \quad \varphi_t(q,w)=(\sigma_w(t),\sigma_w'(t)),
\]
where $w\in T_qM$ and $\sigma_w(t)=\exp_q(tw)$. In terms of the geodesic flow $\varphi_t$ rather than $\gamma$, \eqref{eq2.1} becomes
\begin{equation}\label{eq2.2}
	\begin{aligned}
	&\quad\int_0^af(\pi_1(\varphi_t(p,v)))^2\sin^2\left(\frac{\pi t}a\right){\rm Ric}(\varphi_t(p,v))dt\\
	&\leq(n-1)\frac{\pi^2}{a^2}\int_0^a f(\pi_1(\varphi_t(p,v)))^2\sin^2\left(\frac{\pi t}a\right)dt\\
	&\quad+(n-1)\int_0^a\langle\nabla f,\pi_2(\varphi_t(p,v))\rangle^2\sin^2\left(\frac{\pi t}a\right)dt,
	\end{aligned}
\end{equation}
where $\pi_1:(p,v)\mapsto p$ and $\pi_2:(p,v)\mapsto v$ are projections. This holds for all $(p, v)\in SM$. Now integrate \eqref{eq2.2} over the unit tangent bundle $SM$. The Sasaki metric on $TM$ induces a measure $d\mu$ on $SM$ with $d\mu={\rm dvol}_Md\omega_{n-1}$ called Liouville measure, where $d\omega_{n-1}$ is the volume element of $S^{n-1}$. By the invariance of $d\mu$ under the geodesic flow,  we have
\[
	\int_{SM}h~d\mu=\int_{SM}h\circ\varphi_t~d\mu
\]
for any $h\in C_c^\infty(SM)$ and $t\in\mathbb R$. We now integrate the three terms of \eqref{eq2.2} separately. Integrating the first term over $SM$, we obtain
\begin{align*}
	&\quad\int_{SM}\left(\int_0^af(\pi_1(\varphi_t(p,v)))^2\sin^2\left(\frac{\pi t}a\right){\rm Ric}(\varphi_t(p,v))dt\right)d\mu(p,v)\\
	&=\int_0^a\left(\int_{SM}f(\pi_1(\varphi_t(p,v)))^2{\rm Ric}(\varphi_t(p,v))d\mu(p,v)\right)\sin^2\left(\frac{\pi t}a\right)dt\\
	&=\int_0^a\sin^2\left(\frac{\pi t}a\right)dt\int_{SM}f(\pi_1(p,v))^2{\rm Ric}(p;v)d\mu(p,v)\\
	&=\frac{a}2\int_M\left(\int_{S_{p}M}f(\pi_1(p,v))^2{\rm Ric}(p;v)d\omega_{n-1}(v)\right){\rm dvol}_M(p)\\
	&=\frac{a}2\int_Mf(p)^2\left(\int_{S_{p}M}{\rm Ric}(p;v)d\omega_{n-1}(v)\right){\rm dvol}_M(p)\\
	&=\frac{a\omega_{n-1}}{2n}\int_Mf(p)^2{\rm scal}_M(p)~{\rm dvol}_M(p),
\end{align*}
where the order of integration can be interchanged because $f$ has compact support. Similarly, integrating the second term in \eqref{eq2.2} gives,
\[
	\int_{SM}\left(\int_0^a f(\pi_1(\varphi_t(p,v)))^2\sin^2\left(\frac{\pi t}a\right)dt\right)d\mu(p,v)
	=\frac{a\omega_{n-1}}2\int_Mf(p)^2~{\rm dvol}_M(p).
\]
For the third term of \eqref{eq2.2}, we obtain,
\begin{align*}
	&\quad\int_{SM}\left(\int_0^a\langle\nabla f,\pi_2(\varphi_t(p,v))\rangle^2\sin^2\left(\frac{\pi t}a\right)dt\right)d\mu(p,v)\\
	&=\frac{a}2\int_M\left(\int_{S_{p}M}\langle\nabla f,\pi_2(p,v)\rangle^2d\omega_{n-1}(v)\right){\rm dvol}_M(p)\\
	&=\frac{a\omega_{n-1}}{2n}\int_M|\nabla f(p)|^2~{\rm dvol}_M(p).
\end{align*}
Substituting these gives
\begin{align*}
	&\quad\frac{a\omega_{n-1}}{2n}\int_Mf^2{\rm scal}_M~{\rm dvol}_M\\
	&\leq(n-1)\frac{\pi^2}{2a}\omega_{n-1}\int_Mf^2~{\rm dvol}_M+(n-1)\frac{a\omega_{n-1}}{2n}\int_M|\nabla f|^2~{\rm dvol}_M.
\end{align*}
Simplifying, we have
\[
	\int_Mf^2{\rm scal}_M~{\rm dvol}_M\leq n(n-1)\frac{\pi^2}{a^2}\int_Mf^2~{\rm dvol}_M+(n-1)\int_M|\nabla f|^2~{\rm dvol}_M.
\]
Since ${\rm scal}_M\geq n(n-1)$, we have
\[
	\int_Mf^2{\rm scal}_M~{\rm dvol}_M\geq n(n-1)\int_Mf^2~{\rm dvol}_M,
\]
which yields
\[
	n\int_Mf^2~{\rm dvol}_M\leq n\frac{\pi^2}{a^2}\int_Mf^2~{\rm dvol}_M+\int_M|\nabla f|^2~{\rm dvol}_M,
\]
i.e.,
\[
	\frac{\int_M|\nabla f|^2~{\rm dvol}_M}{\int_Mf^2~{\rm dvol}_M}\geq n\left(1-\frac{\pi^2}{a^2}\right).
\]
Since $a>l$, we have
\[
	n\left(1-\frac{\pi^2}{a^2}\right)>n\left(1-\frac{\pi^2}{l^2}\right)=\lambda_0(M).
\]
Since $f\in C_c^\infty(M)\setminus \{0\}$ was arbitrary, we obtain
\[
	\lambda_0(M)\geq n\left(1-\frac{\pi^2}{a^2}\right)>\lambda_0(M),
\]
which is a contradiction. Hence ${\rm conj}(M)\leq l$.

\end{document}